\documentclass[12pt]{article}
\usepackage{latexsym,amsfonts,amsmath,graphics,amsthm}
\usepackage{epsfig}
\usepackage{mathrsfs}

\newtheorem{theorem}{Theorem}
\newtheorem{lemma}{Lemma}

\newtheorem{definition}{Definition}

\newcommand{\bsalpha}{\boldsymbol{\alpha}}

\newcommand{\bsk}{\boldsymbol{k}}

\newcommand{\bsx}{\boldsymbol{x}}

\newcommand{\bsf}{\boldsymbol{f}}

\newcommand{\bsy}{\boldsymbol{y}}

\newcommand{\icomp}{\mathtt{i}}
\newcommand{\bszero}{\boldsymbol{0}}

\newcommand{\rd}{\,\mathrm{d}}

\newcommand{\NN}{\mathbb{N}}
\newcommand{\ZZ}{\mathbb{Z}}
\newcommand{\QQ}{\mathbb{Q}}
\newcommand{\CC}{\mathbb{C}}

\newcommand{\cS}{\mathcal{S}}
\newcommand{\FF}{\mathbb{Z}}

\newcommand{\RR}{{\mathbb R}}

\newcommand{\walb}{\,_q{\rm wal}}

\makeatletter
\newcommand{\rdots}{\mathinner{\mkern1mu\lower-1\p@\vbox{\kern7\p@\hbox{.}}
\mkern2mu \raise4\p@\hbox{.}\mkern2mu\raise7\p@\hbox{.}\mkern1mu}}
\makeatother

\allowdisplaybreaks

\begin{document}

\title{Metrical lower bounds on the discrepancy of digital Kronecker-sequences}
\author{Gerhard Larcher\thanks{G. L. is partially supported by the Austrian Science Foundation (FWF), project P21943-N18} and Friedrich Pillichshammer}
\date{}
\maketitle

\begin{abstract}
Digital Kronecker-sequences are a non-archimedean analog of classical Kronecker-sequences whose construction is based on Laurent series over a finite field. In this paper it is shown that for almost all digital Kronecker-sequences the star discrepancy satisfies $D_N^\ast \ge c(q,s) (\log N)^s \log \log N$ for infinitely many $N \in \NN$, where $c(q,s)>0$ only depends on the dimension $s$ and on the order $q$ of the underlying finite field, but not on $N$. This result shows that a corresponding metrical upper bound due to Larcher is up to some $\log \log N$ term best possible. 
\end{abstract}

\section{Introduction and statement of the result}

For an $s$ tuple $\bsalpha=(\alpha_1,\ldots,\alpha_s)$ of reals the classical {\it Kronecker-sequence} $\cS(\bsalpha)=(\bsx_n)_{n \ge 0}$ is defined by $$\bsx_n:=(\{n \alpha_1\},\ldots,\{n \alpha_s\})\ \ \mbox{ for }\ \ \ n \in \NN_0,$$ where $\{ x \}$ denotes the fractional part of a real number $x$. It was shown by Weyl \cite{weyl} that that $\cS(\bsalpha)$ is uniformly distributed in the $s$-dimensional unit-cube $[0,1)^s$ if and only if $1,\alpha_1,\ldots,\alpha_s$ are linearly independent over $\QQ$. Quantitative versions of this result can be stated in terms of star discrepancy which is defined as follows:

Let $\cS=(\bsy_n)_{n \ge 0}$ be an infinite sequence in the $s$-dimensional unit-cube $[0,1)^s$. For $\bsx=(x_1,\ldots,x_s) \in [0,1]^s$ and $N \in \NN$ (by $\NN$ we denote the set of positive integers and we set $\NN_0=\NN \cup \{0\}$) the {\it local discrepancy} $D(\bsx,N)$ of $\cS$ is the difference between the number of indices $n=0,\ldots,N-1$ for which $\bsy_n$ belongs to the interval $[\bszero, \bsx)=\prod_{j=1}^s [0,x_j)$ and the expected number $N x_1\cdots x_s$ of points in $[\bszero,\bsx)$ if we assume a perfect uniform distribution on $[0,1]^s$, i.e., $$D(\bsx,N)=\#\{0 \le n < N \ : \ \bsx_n \in [\bszero,\bsx)\}-N x_1 \cdots x_s.$$ 
\begin{definition}[star discrepancy]\rm
The {\it star discrepancy} $D_N^\ast$ of a sequence $\cS$ is the $L_\infty$-norm of the local discrepancy, i.e., $$D_N^\ast(\cS)=\|D(\bsx,N)\|_{\infty}.$$
\end{definition} 

Note that often a normalized version is used for defining the star discrepancy. A sequence $\cS$ is called {\it uniformly distributed} if and only if the normalized star discrepancy $D_N^\ast(\cS)/N$ tends to 0 for growing $N$. Furthermore, the (normalized) star discrepancy can be used to bound the integration error of a quasi-Monte Carlo algorithm based on $\cS$ via the well known Koksma-Hlawka inequality. For more information on uniform distribution, discrepancy and quasi-Monte Carlo integration we refer to \cite{DP10,kuinie,niesiam}.\\

Apart from the one-dimensional case $s=1$, it is very difficult to give good estimates for the star discrepancy of concrete Kronecker-sequences. In a remarkable paper Beck \cite{beck} showed the following metrical result: \\

{\it For arbitrary increasing function $\varphi(n)$ of $n \in \NN$ we have $$D_N^{\ast}(\cS(\bsalpha)) \ll_s (\log N)^s \varphi(\log \log N) \ \ \Leftrightarrow \ \ \sum_{n=1}^\infty \frac{1}{\varphi(n)} < \infty$$ for almost all $\bsalpha \in \RR^s$.}\footnote{Here $A(N,s) \ll_s B(N,s)$ means that there exists a quantity $c(s) > 0$ which depends only on $s$ (and not on $N$) such that $A(N,s) \le c(s) B(N,s)$.} \\

In particular, for almost every $\bsalpha \in \RR^s$ we have $$D_N^{\ast}(\cS(\bsalpha)) \ll_{s,\varepsilon} (\log N)^s (\log \log N)^{1+\varepsilon}$$ for every $\varepsilon >0$, and for almost every $\bsalpha \in \RR^s$ there are infinitely many $N \in \NN$ such that $$D_N^\ast(\cS(\bsalpha)) \ge c(s) (\log N)^s \log \log N$$ with a $c(s)>0$ not depending on $N$. \\

In connection with the construction of digital sequences a ``non-archimedean analog'' to classical Kronecker-sequences has been introduced by Niederreiter~\cite[Section~4]{niesiam} and further investigated by Larcher and Niederreiter~\cite{LN93}.

Let $q$ be a prime number and let $\ZZ_q$ be the finite field of order $q$. We identify $\ZZ_q$ with the set $\{0,1,\ldots,q-1\}$ equipped with arithmetic operations modulo $q$. Let $\ZZ_q[x]$ be the set of all polynomials over $\ZZ_q$ and let $\ZZ_q((x^{-1}))$ be the field of formal Laurent series $$g=\sum_{k=w}^{\infty} a_k x^{-k}\ \ \ \mbox{ with } a_k \in \ZZ_q \ \mbox{ and } \ w \in \ZZ \ \mbox{ with }\ a_w \not=0.$$ The discrete exponential evaluation $\nu$ of $g$ is defined by $\nu(g):=-w$ ($\nu(0):=-\infty$). Furthermore, we define the ``fractional part'' of $g$ by $$\{g\}:= \sum_{k=\max(1,w)}^{\infty} a_k x^{-k}.$$

Throughout the paper we associate a nonnegative integer $n$ with $q$-adic expansion $n=n_0+n_1q+\cdots+n_r q^r$ with the polynomial $n(x)=n_0+n_1 x+\cdots +n_r x^r$ in $\ZZ_q[x]$ and vice versa.

For every $s$-tuple $\bsf=(f_1,\ldots,f_s)$ of elements of $\ZZ_q((x^{-1}))$ we define the sequence $\cS(\bsf)=(\bsx_n)_{n \ge 0}$  by $$\bsx_n=(\{n(x) f_1(x)\}_{|x=q},\ldots,\{n(x) f_s(x)\}_{|x=q})\ \ \mbox{ for }\ \ \ n \in \NN_0.$$ This sequence can be viewed as analog to the classical Kronecker-sequence and is therefore sometimes called a {\it digital Kronecker-sequence} (this terminology will be clearer in a moment). 

In analogy to classical Kronecker-sequences it has been shown in \cite{LN93} that a digital Kronecker-sequence $\cS(\bsf)$ is uniformly distributed in $[0,1)^s$ if and only if $1,f_1,\ldots,f_s$ are linearly independent over $\ZZ_q[x]$. The special case that the $f_i$ are rational functions was studied in \cite{L93} and in \cite{KP12}.

In the analysis of digital Kronecker-sequences one can obviously restrict to the set $\overline{\ZZ}_q((x^{-1}))$ of Laurent series over $\ZZ_q$ with $w \ge 1$, i.e. with $g=\{g\}$.

In analogy to the results of Beck here we are interested in metrical results for the star discrepancy of digital Kronecker-sequences. To tackle this problem we need to introduce a suitable probability measure on $(\overline{\ZZ}_q((x^{-1})))^s$. 

By $\mu$ we denote the normalized Haar-measure on $\overline{\ZZ}_q((x^{-1}))$ and by $\mu_s$ the $s$-fold product measure on $(\overline{\ZZ}_q((x^{-1})))^s$. We remark that $\mu$ has the following rather simple shape: If we identify the elements $\sum_{k=1}^\infty t_k x^{-k}$ of $\overline{\ZZ}_q((x^{-1}))$ where $t_k\not=q-1$ for infinitely many $k$ in the natural way with the real numbers  $\sum_{k=1}^\infty t_k q^{-k}  \in [0,1)$, then, by neglecting the countable many elements where $t_k \not= q-1$ only for finitely many $k$, $\mu$ corresponds to the Lebesgue measure $\lambda$ on $[0,1)$. For example, the ``cylinder set'' $C(c_1,\ldots,c_m)$ consisting of all elements $g=\sum_{k=1}^\infty a_k x^{-k}$ from $\overline{\ZZ}_q((x^{-1}))$ with $a_k=c_k$ for $k=1,\ldots,m$ and arbitrary $a_k \in \ZZ_q$ for $k \ge m+1$ has measure $\mu( C(c_1,\ldots,c_m) )=q^{-m}$. 

In \cite{L95} Larcher proved the following metrical upper bound on the star discrepancy of digital Kronecker-sequences.

\begin{theorem}[Larcher, 1995]\label{th0}
Let $s \in \NN$, let $q$ be a prime number and let $\varepsilon>0$. For $\mu_s$-almost all $\bsf \in (\overline{\ZZ}_q((x^{-1})))^s$ the digital Kronecker-sequence $\cS(\bsf)$ has star discrepancy satisfying $$D_N^\ast(\cS(\bsf)) \le c(q,s,\varepsilon) (\log N)^s (\log\log N)^{2+\varepsilon}$$ with a $c(q,s,\varepsilon)>0$ not depending on $N$.
\end{theorem}

Recall that it follows from a result of Roth~\cite{roth} that there exists a quantity $c(s)>0$ such that for every sequence $\cS$ in $[0,1)^s$ we have 
\begin{equation}\label{bdroth}
 D_N^{\ast}(\cS(\bsf)) \ge c(s) (\log N)^{s/2} \ \ \ \mbox{ for infinitely many }\ \ N \in \NN.
\end{equation}
For a proof, see, for example, \cite[Chapter~2, Theorem~2.2]{kuinie}. Many people believe that the exponent $s/2$ of the logarithm in \eqref{bdroth} can be replaced by $s$ but until now there is no proof of this conjecture for $s\ge 2$. For $s=1$ we have 
\begin{equation}
D_N^{\ast}(\cS) \ge c \log N \ \ \ \mbox{ for infinitely many }\ \ N \in \NN
\end{equation}
with a constant $c>0$ which is independent of $N$. This has been shown by Schmidt~\cite{Schm72distrib}.

It is the object of this paper to show that the metrical upper bound from Theorem~\ref{th0} is best possible in the order of magnitude in $N$ (up to some $\log\log N$ term). We will prove:

\begin{theorem}\label{th1}
Let $s \in \NN$ and let $q$ be a prime number. For $\mu_s$-almost all $\bsf \in (\overline{\ZZ}_q((x^{-1})))^s$ the digital Kronecker-sequence $\cS(\bsf)$ has star discrepancy satisfying $$D_N^\ast(\cS(\bsf)) \ge c(q,s) (\log N)^s \log\log N \ \ \mbox{ for infinitely many } N \in \NN$$ with some $c(q,s)>0$ not depending on $N$.
\end{theorem}

For the proof of Theorem~\ref{th1} we use an approach similar to the technique used by Beck \cite{beck} to give a metric lower bound for the discrepancy of Kronecker sequences. In the following section we will collect some auxiliary results. The proof of Theorem~\ref{th1} is then presented in Section~\ref{proofthm}.

\section{Auxiliary results}

For given $\bsf=(f_1,\ldots,f_s) \in (\overline{\ZZ}_q((x^{-1})))^s$ with $f_j=\frac{f_{j,1}}{x}+\frac{f_{j,2}}{x^2}+\frac{f_{j,3}}{x^3}+\cdots \in \overline{\ZZ}_q((x^{-1}))$ we define $\NN \times \NN$ matrices $C_1,\ldots,C_s$ over $\ZZ_q$ by $$C_j=\left(
\begin{array}{llll}
f_{j,1} & f_{j,2} & f_{j,3} & \ldots\\
f_{j,2} & f_{j,3} & f_{j,4} & \ldots\\
f_{j,3} & f_{j,4} & f_{j,5} & \ldots\\ 
\multicolumn{4}{c}\dotfill 
\end{array}\right).$$
Then the elements $\bsx_n$ of the digital Kronecker-sequence can be constructed with the following digital method: For $n \in \NN_0$ with $q$-adic expansion $n = n_0 + n_1 q + n_2 q^2+\cdots $ (this expansion is obviously finite) we set $$\vec{n} = (n_0, n_1, n_2, \ldots )^\top \in (\ZZ_q^{\mathbb{N}})^\top$$ and then we put  
\begin{equation*}
\vec{x}_{n,j} := C_j \vec{n} \quad \mbox{for } j = 1,\ldots, s
\end{equation*}
where all arithmetic operations are taken modulo $q$. Write $\vec{x}_{n,j}$ as $\vec{x}_{n,j} = (x_{n,j,1}, x_{n,j,2},\ldots)^\top$. Then the $n$th point $\boldsymbol{x}_n$ of the sequence $\cS(\bsf)$ is given by $\boldsymbol{x}_n = (x_{n,1}, \ldots, x_{n,s})$ where 
\begin{equation*}
x_{n,j} = x_{n,j,1} q^{-1} + x_{n,j,2} q^{-2} + \cdots.
\end{equation*}
It follows that digital Kronecker-sequences are just special examples of digital sequences as introduced by Niederreiter in \cite{nie87}, see also \cite{DP10,niesiam}. This way of describing the sequence $\cS(\bsf)$ is the reason why it is called a digital Kronecker-sequence. \\

We continue with some notational issues. As already mentioned we sometimes consider $j\in \NN_0$ as elements in $\ZZ_q[x]$ and vice versa. Similarly, $f \in \overline{\ZZ}_q((x^{-1}))$ is sometimes considered as element in $[0,1)$ and vice versa, just by substituting $q$ for $x$. It should always be clear from the context what is meant. However, multiplication and addition of polynomials and Laurent series are always performed in $\ZZ_q((x^{-1}))$.\\

An important tool in our analysis are $q$-adic Walsh functions which we introduce now:

\begin{definition}[$q$-adic Walsh functions]\rm
Let $q$ be a prime number and let $\omega_q:=\exp(2 \pi \icomp/q)$ be the $q$th root of unity. For $j \in \NN_0$ with $q$-adic expansion $j=j_0+j_1 q+j_2 q^2+\cdots$ (this expansion is obviously finite) the $j$th $q$-adic {\it Walsh function} $\walb_j:\RR \rightarrow \CC$, periodic with period one, is defined as $$\walb_j(x)=\omega_q^{j_0 \xi_1+j_1 \xi_2+j_2 \xi_3+\cdots}$$ whenever $x \in [0,1)$ has $q$-adic expansion of the form $x=\xi_1 q^{-1}+\xi_2 q^{-2}+\xi_3 q^{-3}+\cdots$ (unique in the sense that infinitely many of the digits $\xi_i$ must be different from $q-1$).
\end{definition}

We collect some properties of Walsh functions. More informations can be found in \cite[Appendix~A]{DP10}.

\begin{lemma}\label{le1}
For $j,k,l \in \ZZ_q[x]$ and $f,f_1,f_2 \in \ZZ_q((x^{-1}))$ we have 
\begin{enumerate}
\item $\walb_j(k f_1+l f_2)= \walb_{j k}(f_1) \walb_{jl}(f_2)$, where $k f_1+l f_2$ is evaluated in $\ZZ_q((x^{-1}))$ and $jk$ and $jl$, respectively, in $\ZZ_q[x]$; 
\item $\walb_j(f) \walb_k(f)=\walb_{j+k}(f)$;
\item \label{ass2} $$\sum_{x=0}^{q^m-1} \walb_k(x/q^m) \overline{\walb_l(x/q^m)}=\left\{
\begin{array}{ll}
1 &  \mbox{ if } x^m | k-l, \\
0 & \mbox{ otherwise},
\end{array}\right.$$ (orthonormality of Walsh functions); and
\item $\int_0^1 \walb_k(x) \rd x=0$ whenever $k \not=0$.
\end{enumerate}
\end{lemma}

\begin{proof}
These are standard properties of Walsh functions and are easily deduced from their definition. Alternatively we refer to \cite[Appendix~A]{DP10}. 
\end{proof}

We need some notation. For $m \in \NN_0$ let 
\begin{eqnarray*}
\QQ(q^m) & = & \{x=r q^{-m} \in [0,1) \, : \, r=0,\ldots,q^m -1\},\\
\QQ^s(q^m) & = & \{\bsx=(x_1,\ldots,x_s) \in [0,1)^s \, : x_j \in \QQ(q^m) \ \mbox{ for }\  j=1,\ldots,s\}.
\end{eqnarray*}

\begin{lemma}\label{le6}
Let $\cS(\bsf)=(\bsx_n)_{n \ge 0}$ be a digital Kronecker-sequence generated by an $s$-tuple $\bsf=(f_1,\ldots,f_s) \in (\overline{\ZZ}_q((x^{-1})))^s$.  Let $N \in \NN$ with base $q$ expansion $N=N_{m-1}q^{m-1}+\cdots +N_1 q + N_0$ and let $\bsx=(x_1,\ldots,x_s) \in \QQ^s(q^m)$. Then we have
\begin{eqnarray*}
D(\bsx,N) =  \sum_{k_1,\ldots, k_s=0\atop (k_1,\ldots,k_s)\not=(0,\ldots ,0)}^{q^m-1} \left(\prod_{j=1}^s J_{k_j}(x_j)\right) G(N,w(k_1,\ldots,k_s)),  
\end{eqnarray*}
where for $k=\kappa q^{a-1}+k'$ with $a \in \NN$, $1 \le
\kappa <q$ and $0 \le k' < q^{a-1}$ we have 
\begin{eqnarray}\label{valJk}
J_k(x) & = & \frac{1}{q^a} \Bigg(\frac{1}{1-\omega_q^{-\kappa}}\ 
\overline{\walb_{k'}(x)} + \left(\frac{1}{2} +
\frac{1}{\omega_q^{-\kappa}-1}\right) \ \overline{\walb_k(x)} \\ &&
\qquad + \sum_{c=1}^{m-a} \sum_{l = 1}^{q-1} \frac{1}{q^c
(\omega_q^l -1)} \ \overline{\walb_{l q^{a+c-1}+k}(x)}- \frac{1}{2 q^{m-a}} \ \overline{\walb_k(x)} \Bigg)
\end{eqnarray}
and for $k = 0$ we have
\begin{equation}\label{valJ0}
J_0(x) = \frac{1}{2} + \sum_{c=1}^m \sum_{l=1}^{q-1}
\frac{1}{q^c (\omega_q^l-1)}\ \overline{\walb_{l q^{c-1}}(x)}-\frac{1}{2 q^m},\nonumber
\end{equation} 
and where $$G(N,w(k_1,\ldots,k_s))=\left[\omega_q^{b_{w+1} N_{w+1}+\cdots +b_{m-1} N_{m-1}} q^w \left(\frac{\omega_q^{b_w N_w}-1}{\omega_q^{b_w}-1} + \omega_q^{b_w N_w} \left\{\frac{N}{q^w}\right\}\right)\right]$$ with $$w=w(k_1,\ldots,k_s) =-\nu\left(\left\{\sum_{j=1}^s k_j f_j\right\} \right).$$ However, if $w \ge m$, then we put $G(N,w)=N$.
\end{lemma}

\begin{proof}
This follows directly from \cite[Lemma~7]{LP13} and the construction of $\cS(\bsf)$ in terms of matrices $C_1,\ldots,C_s$.
\end{proof}

\begin{lemma}\label{le2}
For $m \in \NN$ and $k_1,\ldots,k_s \in \ZZ_q[x]$, not all of them 0, let $$M_m(k_1,\ldots,k_s):=\{(f_1,\ldots,f_s) \in (\overline{\ZZ}_q((x^{-1}))^s  \ : \ \nu(\{k_1f_1+\cdots +k_s f_s\}) \le -m\}.$$ Then we have $$\mu_s( M_m(k_1,\ldots,k_s) ) = \frac{1}{q^{m-1}}.$$
\end{lemma}

\begin{proof}
Let $\chi$ be the characteristic function of the interval $[0,q^{-(m-1)})$. Then $\chi$ has a finite Walsh series representation in base $q$ of the form $$\chi(x)=\sum_{i=0}^{q^{m-1} -1} a_i \walb_i(x)$$ with $a_0=q^{-(m-1)}$, see \cite[Lemma~3.9]{DP10}. Now
\begin{eqnarray*}
\mu_s( M_m(k_1,\ldots,k_s) ) & = & \int_{[0,1]^s} \chi(\{k_1f_1+\cdots +k_s f_s\}) \rd f_1 \ldots \rd f_s\\
& = & a_0 + \sum_{i=1}^{q^{m-1}-1} a_i  \int_{[0,1]^s} \walb_i(k_1f_1+\cdots +k_s f_s) \rd f_1 \ldots \rd f_s\\ 
& = & \frac{1}{q^{m-1}} + \sum_{i=1}^{q^{m-1}-1} a_i \prod_{j=1}^s  \int_0^1 \walb_{i k_j}(f_j) \rd f_j\\
& = & \frac{1}{q^{m-1}},
\end{eqnarray*}
since at least one of the $k_j$ is different from zero and for such a $k_j$ we have $\int_0^1 \walb_{i k_j}(f_j) \rd f_j=0$ according to Lemma~\ref{le1}.
\end{proof}

\begin{lemma}\label{le3}
Let $\overline{P} \subseteq (\ZZ_q[x]\setminus \{0\})^s$. For $(k_1,\ldots,k_s) \in \overline{P}$ with $\deg(k_j)=r_j$ for $j=1,\ldots,s$ let $\beta_1(k_1),\ldots,\beta_s(k_s) \in \ZZ_q[x]$ be polynomials which satisfy $\beta_j(k_j)=0$ or $\gcd(\beta_j(k_j),k_j)=1$ for all $j=1,\ldots ,s$, but not all of them equal to zero. Let 
\begin{eqnarray*}
\widetilde{M} & := & \{ (f_1,\ldots,f_s)\in (\overline{\ZZ}_q((x^{-1})))^s \ : \ \nu(\{k_1f_1+\cdots +k_s f_s\}) \le  -(r_1+\cdots +r_s), \\
& & \hspace{0.5cm} \nu(\{(k_1+\beta_1(k_1))f_1+\cdots +(k_s+\beta_s(k_s))f_s\}) \le - \lfloor (r_1+\cdots +r_s)/2 \rfloor\\
& & \hspace{0.5cm} \mbox{for infinitely many } \ (k_1,\ldots,k_s) \in \overline{P}\ \mbox{ with }\ \gcd(k_1,\ldots,k_s)=1\}.
\end{eqnarray*}
Then we have $$\mu_s(\widetilde{M})=0.$$
\end{lemma}

\begin{proof}
For given $k_1,\ldots,k_s$ with $\gcd(k_1,\ldots,k_s)=1$ and not all $k_j=1$ let 
$$M_1(k_1,\ldots,k_s) := \{(f_1,\ldots,f_s)\, : \, \nu(\{k_1f_1+\cdots +k_s f_s\}) \le - (r_1+\cdots +r_s)\}$$ and 
\begin{eqnarray*}
M_2(k_1,\ldots,k_s) & := & \bigg\{(f_1,\ldots,f_s)\, : \\
& &  \, \nu(\{(k_1+\beta_1(k_1))f_1+\cdots +(k_s+\beta_s(k_s))f_s\}) \le - \left\lfloor \frac{r_1+\cdots+r_s}{2}\right\rfloor \bigg\}.
\end{eqnarray*}  
Let $m=r_1+\cdots+r_s$, let $\chi_1$ be the characteristic function of the interval $[0,q^{-(m-1)})$ and let $\chi_2$ be the characteristic function of the interval $[0,q^{\lfloor m/2 \rfloor-1})$. Then we have finite Walsh series representation of $\chi_1$ and $\chi_2$ in base $q$ given by
$$\chi_1=\sum_{i=0}^{q^{m-1}-1} a_i^{(1)} \walb_i \ \ \mbox{ and }\ \ \chi_2=\sum_{i=0}^{q^{\lfloor m/2 \rfloor -1} -1} a_i^{(2)} \walb_i$$ with $$a_0^{(1)}=\frac{1}{q^{m-1}} \ \ \ \mbox{ and }\ \ \ a_0^{(2)}=\frac{1}{q^{\left\lfloor m/2\right\rfloor -1}}.$$ Now we have
\begin{eqnarray*}
\lefteqn{\mu_s( M_1(k_1,\ldots,k_s) \cap M_2(k_1,\ldots,k_s) )}\\
& = & \int_{[0,1]^s} \chi_1(\{k_1f_1+\cdots +k_s f_s\}) \chi_2(\{(k_1+\beta_1(k_1))f_1+\cdots +(k_s+\beta_s(k_s))f_s\}) \rd f_1 \ldots \rd f_s\\
& = & a_0^{(1)} a_0^{(2)} + \sum_{i,j \atop (i,j)\not=(0,0)} a_i^{(1)} a_j^{(2)} \int_{[0,1]^s} \walb_i(k_1f_1+\cdots +k_s f_s)\\
& & \hspace{4cm} \times \walb_j((k_1+\beta_1(k_1))f_1+\cdots +(k_s+\beta_s(k_s))f_s) \rd f_1 \ldots \rd f_s.
\end{eqnarray*}
The integral in the last sum equals $$\prod_{l=1}^s \int_0^1 \walb_{ik_l+j(k_l+\beta_l(k_l))}(f_l) \rd f_l$$ and this is zero unless we have 
\begin{equation}\label{glstern1}
ik_l+j(k_l+\beta_l(k_l)) = 0 \ \ \ \ \forall\ l=1,\ldots,s. 
\end{equation}
This certainly cannot hold if $i=0$ or $j=0$. Let $i,j \not=0$. If $\beta_{\overline{l}}(k_{\overline{l}})=0$ for some $\overline{l}$ and if \eqref{glstern1} holds, then $i k_{\overline{l}}+j k_{\overline{l}}=0$ and hence $i+j=0$. Therefore we have $j \beta_l(k_l)=0$ for all $l=1,\ldots ,s$ and hence $\beta_l(k_l)=0$ for all $l=1,\ldots,s$ what is a contradiction. This means: If \eqref{glstern1} holds, then $i,j \not=0$ and $\beta_l(k_l)\not=0$ for all $l=1,\ldots,s$. Hence for all $l=1,\ldots,s$ we have $\gcd(k_l,k_l+\beta_l(k_l))=1.$

Now, if \eqref{glstern1} holds, for any $l,l'$ with $l \not= l'$ we have $$i k_l+j(k_l+\beta_l(k_l))=0 \ \ \mbox{ and }\ \ i k_{l'}+j(k_{l'}+\beta_{l'}(k_{l'}))=0,$$ hence
\begin{eqnarray*}
0 & = & i k_l k_{l'}+j(k_{l'}+j(k_l+\beta_l(k_l)) k_{l'}\\
& = & -k_l j (k_{l'}+\beta_{l'}(k_{l'}))+j(k_l+\beta_l(k_l)) k_{l'}\\
& = & j(\beta_l(k_l) k_{l'}-\beta_{l'}(k_{l'}) k_l), 
\end{eqnarray*}
and therefore $$\beta_{l'}(k_{l'}) k_l=\beta_l(k_l) k_{l'}.$$ Since $\gcd(k_l,\beta_l)=1$, we conclude that $k_l | k_{l'}$ for all $l'$. Since $\gcd(k_1,\ldots,k_s)=1$ it follows that $k_l=1$ and, since $l$ was arbitrary, $(k_1,\ldots,k_s)=(1,\ldots,1)$, a contradiction to the assumptions. So $$\mu_s( M_1(k_1,\ldots,k_s) \cap M_2(k_1,\ldots,k_s) )=a_0^{(1)} a_0^{(2)}= \frac{1}{q^{m+ \lfloor m/2 \rfloor -2}}$$ and $$\mu_s(\widetilde{M}) \le \lim_{R \rightarrow \infty} \sum_{r_1,\ldots,r_s \atop r_1+\cdots+r_s \ge R} \sum_{k_1,\ldots,k_s \atop \deg(k_i)=r_i} \frac{1}{q^{r_1+\cdots+r_s+ \lfloor (r_1+\cdots+r_s)/2 \rfloor -2}}=0.$$
\end{proof}

\begin{lemma}\label{le5}
Let $(\Omega, \mathcal{A},\mu)$ be a measure space and let $(A_n)_{n \ge 1}$ be a sequence of sets $A_n \in \mathcal{A}$ such that $$\sum_{n=1}^\infty \mu(A_n)=\infty.$$ Then the set $A$ of points falling in infinitely many sets $A_n$ is of measure $$\mu(A) \ge \limsup_{Q \rightarrow \infty} \frac{\left(\sum_{n=1}^Q \mu(A_n)\right)^2}{\sum_{n,m=1}^Q \mu(A_n \cap A_m)}.$$ 
\end{lemma}

\begin{proof}
This is \cite[Lemma~5 in Chapter I]{Sp}. A proof can be found there. 
\end{proof}

\begin{lemma}\label{le4}
Let $P \subseteq (\ZZ_q[x]\setminus\{0\})^s$ such that $(1,\ldots,1) \not\in P$ and for each $(k_1,\ldots,k_s) \in P$ we have $\gcd(k_1,\ldots,k_s)=1$. Let 
\begin{eqnarray*}
\overline{M}=\{(f_1,\ldots,f_s)\in (\ZZ_q((x^{-1})))^s & : & \nu(\{k_1f_1+\cdots +k_s f_s\}) \le -F(r_1,\ldots,r_s)\\
& & \mbox{ for infinitely many } (k_1,\ldots,k_s)\in P\},
\end{eqnarray*}
where $r_i=\deg(k_i)$, and $F: \NN_0^s \rightarrow \NN$ is such that $$\sum_{(k_1,\ldots,k_s) \in P} \frac{1}{q^{F(r_1,\ldots,r_s)}}=\infty.$$ Then $$\mu_s(\overline{M})=1.$$
\end{lemma}

\begin{proof}
For given  $(k_1,\ldots,k_s) \in P$ let $$M(k_1,\ldots,k_s):=\{(f_1,\ldots,f_s) \in (\ZZ_q((x^{-1})))^s \, : \, \nu(\{k_1f_1+\cdots+k_s f_s\}) \le -F(r_1,\ldots,r_s)\}.$$  With the same proof as for Lemma~ \ref{le2} we have $$\mu_s(M(k_1,\ldots,k_s))=\frac{1}{q^{F(r_1,\ldots,r_s)-1}}$$ and hence $$\sum_{(k_1,\ldots,k_s) \in P} \mu_s(M(k_1,\ldots,k_s)) = q \sum_{(k_1,\ldots,k_s) \in P} \frac{1}{q^{F(r_1,\ldots,r_s)}}=\infty,$$ and we can use Lemma~\ref{le5} to obtain 
\begin{equation}\label{ssp1}
\mu_s(\overline{M}) \ge \lim_{R \rightarrow \infty} \frac{\left( \sum\limits_{(k_1,\ldots,k_s)\in P\atop r_1+\cdots +r_s \le R} \mu_s(M(k_1,\ldots,k_s))\right)^2}{\sum\limits_{(k_1,\ldots,k_s)\in P \atop {(l_1,\ldots ,l_s)\in P \atop \sum \deg(k_j), \sum \deg(l_j) \le R}} \mu_s(M(k_1,\ldots,k_s) \cap M(l_1,\ldots,l_s))}.
\end{equation}

Proceeding in the same way as in the proof of Lemma~\ref{le3} we obtain $$\mu_s(M(k_1,\ldots,k_s) \cap M(l_1,\ldots, l_s))=\mu_s(M(k_1,\ldots,k_s)) \mu_s(M(l_1,\ldots,l_s))$$ provided that for all $(i,j) \not=(0,0)$ we have that 
\begin{equation}\label{glstern2}
i k_u +j l_u=0 
\end{equation}
does not hold for all $u=1,\ldots ,s$. Of course, by the definition of $P$ the condition \eqref{glstern2} can only hold if $i,j \not=0$. If for $u \not=v$ we have $$i k_u+j l_u =0 \ \ \ \mbox{ and }\ \ \ i k_v+j l_v=0,$$ then 
$$0 = i k_u k_v+j l_u k_v= k_u(-j l_v)+j l_u k_v= j(l_u k_v-k_u l_v).$$
Hence, if \eqref{glstern2} holds for all $u$, then we have $$l_u k_v=k_u l_v$$ for all $u,v =1,\ldots,s$ what is a contradiction since $\gcd(k_1,\ldots,k_s)=\gcd(l_1,\ldots,l_s)=1$ and $(k_1,\ldots,k_s)\not=(1,\ldots,1)$ and $(l_1,\ldots,l_s)\not=(1,\ldots,1)$ unless $(k_1,\ldots,k_s)=(l_1,\ldots,l_s)$.

If we denote the summands of the sum in the denominator of \eqref{ssp1} in any order by $a_1,a_2,\ldots,a_Q$, then the expression on the right hand side of \eqref{ssp1} can be written as 
\begin{equation}\label{lim1}
\lim_{Q \rightarrow \infty} \frac{\left(\sum_{k=1}^Q a_k\right)^2}{\left(\sum_{k=1}^Q a_k\right)^2+\sum_{k=1}^Q a_k - \sum_{k=1}^Q a_k^2}.
\end{equation}
Since $0 \le a_k \le 1$ for all $k$, and since $\lim_{Q \rightarrow \infty} \sum_{k=1}^Qa_k=\infty$ the limit in \eqref{lim1} is one and the result follows.
\end{proof}

\section{The proof of Theorem~\ref{th1}}\label{proofthm}

We use the representation of $D(\bsx,N)$ given in Lemma~\ref{le6}. For any $\bsk^\ast=(k_1^\ast,\ldots,k_s^\ast) \in \NN^s$ with the property that each of the $k_i^\ast$ is of the form $$k_i^\ast=q^{a_i^\ast -1}+q^{a_i^\ast-2}+l_i^\ast$$ with some $a_i^\ast \ge 3$ and some $0 \le l_i^\ast < q^{a_i^\ast -2}$ we put
\begin{eqnarray*}
\Lambda & := & \Lambda(\bsk^\ast) := \sum_{\bsx \in \QQ^s(q^m)} D(\bsx,N) \walb_{\bsk^\ast}(\bsx)\\
& = & \sum_{k_1,\ldots,k_s =0 \atop (k_1,\ldots,k_s)\not=(0,\ldots,0)}^{q^m -1} \left[ \prod_{j=1}^s \sum_{x_j \in \QQ(q^m)} J_{k_j}(x_j) \walb_{k_j^\ast}(x_j) \right] G(N,w(k_1,\ldots,k_s)).
\end{eqnarray*}

By the definition of the $J_k$ and by the orthonormality of Walsh functions (see Lemma~\ref{le1}) we have $$\theta(k):=\sum_{x \in \QQ(q^m)} J_k(x) \walb_{k^\ast}(x)=0$$ unless we are in one of the following three cases (with $k=\kappa q^{a-1}+k'$ and $k^\ast=q^{a^\ast -1}+(k^\ast)'$):
\begin{enumerate}
\item \label{cs1} $k$ is such that $k'=k^\ast$, i.e., $k=\kappa q^{a^\ast +c-1}+k^\ast$ for some $c \in \NN$ and $\kappa \in \{1,\ldots,q-1\}$. In this case we have  $$\theta(k)= \frac{1}{q^{a^\ast +c}} \frac{1}{1-\omega_q^{-\kappa}}.$$
\item \label{cs2} $k$ is such that $k=k^\ast$. In this case we have  $$\theta(k)=\frac{1}{q^{a^\ast}}\left(\frac{1}{2}+\frac{1}{\omega_q -1}\right)-\frac{1}{2 q^m}.$$
\item \label{cs3} $k$ is such that $k=(k^\ast)'=q^{a^\ast-2}+l^\ast$. In this case we have  $$\theta(k)=\frac{1}{q^{a^\ast}} \frac{1}{\omega_q-1}.$$
\end{enumerate}

We write $(k_j^\ast)'=:\widetilde{k}_j=q^{a_j^\ast -2}+\cdots$ (note that $k_j^\ast$ is uniquely determined by $\widetilde{k}_j$),
\begin{eqnarray*}
\beta_j(\widetilde{k}_j,0) & := & 0\\
\beta_j(\widetilde{k}_j,1) & := & q^{a_j^\ast -1}
\end{eqnarray*}
and for $t \in \NN_0$ and $u_j \in \{tq-t+2,\ldots,tq-t+q\}$ we put $$\beta_j(\widetilde{k}_j,u_j)=q^{a_j^\ast -1}+q^{a_j^\ast+t}(u_j-(tq-t+1)).$$ Then for $u_j \ge 2$ we have $$\widetilde{k}_j+\beta_j(\widetilde{k}_j,u_j)) = (k_j^\ast)' + q^{a_j^\ast -1}+q^{a_j^\ast+t}(u_j-(tq-t+1))=k_j^\ast + q^{a_j^\ast +t} \kappa,$$ where $\kappa=u_j-(tq-t+1)$. Hence, according to Case~\ref{cs1}, we have $$|\theta(\widetilde{k}_j+\beta_j(\widetilde{k}_j,u_j)))|\le c_1(q) \frac{1}{q^{a_j^\ast +t+1}} \le c_1(q) \frac{1}{q^{a_j^\ast +u_j/q}} $$ for some $c_1(q)>0$. Similarly, $$\widetilde{k}_j+\beta_j(\widetilde{k}_j,0)) = (k_j^\ast)' $$ and hence, according to Case~\ref{cs3}, we have $$|\theta(\widetilde{k}_j+\beta_j(\widetilde{k}_j,0)))|\le c_2(q) \frac{1}{q^{a_j^\ast}}$$ for some $c_2(q)>0$, and $$\widetilde{k}_j+\beta_j(\widetilde{k}_j,1)) = (k_j^\ast)' +q^{a_j^\ast -1}=k_j^\ast$$ and hence, according to Case~\ref{cs2}, we have $$|\theta(\widetilde{k}_j+\beta_j(\widetilde{k}_j,1)))|\le c_3(q) \frac{1}{q^{a_j^\ast}}$$ for some $c_3(q)>0$. Summing up, for all $u_j \ge 0$ we have 
\begin{equation}\label{inequjtheta}
|\theta(\widetilde{k}_j+\beta_j(\widetilde{k}_j,u_j)))|\le c_4(q) \frac{1}{q^{a_j^\ast +u_j/q}}
\end{equation}
for some $c_4(q)>0$.

Now we have $$\Lambda=\sum_{u_1,\ldots,u_s \ge 0} \left[\prod_{j=1}^s \theta(\widetilde{k}_j+\beta_j(\widetilde{k}_j,u_j))\right] G(N,w(\widetilde{k}_1+\beta_1(\widetilde{k}_1,u_1),\ldots,\widetilde{k}_s+\beta_s(\widetilde{k}_s,u_s)))$$ where the summation is over all $u_j$ with $\widetilde{k}_j+\beta_j(\widetilde{k}_j, u_j) < q^m$ for all $j=1,\ldots,s$. Then for any $J \in \NN$ we have
\begin{eqnarray}\label{defLambda}
|\Lambda| & \ge & \left|\left[\prod_{j=1}^s \theta(\widetilde{k}_j)\right] G(N,w(\widetilde{k}_1,\ldots,\widetilde{k}_s))\right|\\
&&-\left|\sum_{0 \le u_1,\ldots,u_s \le J \atop (u_1,\ldots,u_s)\not=(0,\ldots,0)} \left[\prod_{j=1}^s \theta(\widetilde{k}_j+\beta_j(\widetilde{k}_j,u_j))\right] G(N,w(\widetilde{k}_1+\beta_1(\widetilde{k}_1,u_1),\ldots,\widetilde{k}_s+\beta_s(\widetilde{k}_s,u_s)) \right| \nonumber \\
&&-\left|\sum_{u_1,\ldots,u_s \ge 0 \atop \exists j:\ u_j >J} \left[\prod_{j=1}^s \theta(\widetilde{k}_j+\beta_j(\widetilde{k}_j,u_j))\right] G(N,w(\widetilde{k}_1+\beta_1(\widetilde{k}_1,u_1),\ldots,\widetilde{k}_s+\beta_s(\widetilde{k}_s,u_s)) \right|.\nonumber 
\end{eqnarray}

Note that $|G(N,w(k_1,\ldots,k_s))| \le q N$ always. Therefore and using \eqref{inequjtheta} for the last sum in \eqref{defLambda} we have 
\begin{eqnarray*}
\left| \Sigma \right| \le \frac{q N}{q^{a_1^\ast + \cdots + a_s^\ast}} \sum_{u_1,\ldots,u_s \ge 0 \atop \exists j:\ u_j >J} q^{-\frac{u_1}{q}-\cdots - \frac{u_s}{q}}\le c_5(q,s) \frac{N}{q^{a_1^\ast + \cdots + a_s^\ast}} \frac{1}{q^{J/q}},
\end{eqnarray*}
with some $c_5(q,s)>0$ depending only on $q$ and on $s$.\\

Let the function $F$ from Lemma~\ref{le4} be such that 
\begin{equation}\label{defF}
q^{F(r_1,\ldots,r_s)} = q^{r_1+\cdots+r_s} (r_1+\cdots +r_s)^s \log (r_1+\cdots+r_s).
\end{equation}
Let $P$ from Lemma~\ref{le4} be given by 
\begin{eqnarray}\label{defP}
P & = & \bigg\{(k_1,\ldots,k_s) \in (\ZZ_q[x]\setminus \{0\})^s \ : \ (k_1,\ldots,k_s)\not=(1,\ldots,1),\ \gcd(k_1,\ldots,k_s)=1 \nonumber\\
& &\hspace{0.5cm} k_i=q^{a_i-1}+\ell_i \mbox{ for some } a_i \in \NN \mbox{ and } 0 \le \ell_i < q^{a_i -1} \mbox{ for all } i=1,\ldots, s\nonumber\\
& &\hspace{0.5cm} \mbox{ and } \gcd\bigg(k_i,x \prod_{j=1}^J \prod_{\kappa=1}^{q-1} (1+\kappa x^j)\bigg)=1 \mbox{ for all } i=1,\ldots,s\bigg\}.
\end{eqnarray}


\begin{lemma}
With $F$ as in \eqref{defF} and $P$ as in \eqref{defP} we have $$\sum_{(k_1,\ldots,k_s)\in P} \frac{1}{q^{F(r_1,\ldots,r_s)}}=\infty,$$ where $r_i =\deg(k_i)$.
\end{lemma}

\begin{proof}
We put $$T:= \sum_{(k_1,\ldots,k_s)\in P} \frac{1}{q^{F(r_1,\ldots,r_s)}}.$$ Let $W_q(a)$ be the set of all monic polynomials over $\FF_q$ with degree $a$, i.e. $$W_q(a)=\{k \in \FF_q[x]\ : \ \deg(k)=a \ \mbox{ and } \ k \ \mbox{ is monic}\}.$$ Put $p:=x \prod_{j=1}^J \prod_{\kappa=1}^{q-1} (1+\kappa x^j)=u p_1^{\alpha_1} \cdots p_r^{\alpha_r}$ with $u \in \ZZ_q$, irreducible factors $p_1,\ldots,p_r\in \ZZ_q[x]$ and $\alpha_1,\ldots,\alpha_r \in \NN$. Then we have
\begin{eqnarray}\label{sumT}
T & = & \sum_{a_1,\ldots ,a_s \atop a_1+\cdots +a_s\not =0}\frac{1}{q^{F(a_1,\ldots,a_s)}} \underbrace{\sum_{k_1 \in W_q(a_1)\atop \gcd(k_1,p)=1} \ldots  \sum_{k_s \in W_q(a_s)\atop \gcd(k_s,p)=1}}_{\gcd(k_1,\ldots,k_s)=1} 1.
\end{eqnarray}

Let $\mu_q$ the polynomial analog to the M\"obius-$\mu$ function defined by $\mu_q(a)=1$ and $\mu_q(a f)=\mu(f)$ for $a \in \FF_q$ and $f \in \FF_q[x]$, $\mu_q(f)=0$ if there exists an irreducible $g \in \FF_q[x]$ with $g^2|f$ and $\mu_q(f)=(-1)^{\rho}$ if $f$ splits up in $\rho$ different irreducible factors. We just remark that $\mu_q$ is multiplicative and refer to \cite[p.~42]{car} for more informations.

First we consider the case $s=1$. Then the inner sum in \eqref{sumT} reduces to (we omit the index ``1'' for the sake of simplicity)
\begin{eqnarray*}
\sum_{k \in W_q(a)\atop \gcd(k,p)=1} 1 & = & \sum_{k \in W_q(a)} \sum_{\ell | \gcd(k,p)} \mu_q(\ell)\\
& = & \sum_{\ell | p} \mu_q(\ell) \sum_{k \in W_q(a) \atop \ell | k} 1\\
& = & \sum_{\ell | p} \mu_q(\ell) \sum_{c \in \FF_q[x] \atop \ell c \in W_q(a)}1.
\end{eqnarray*}
If $\ell c \in W_q(a)$, then the leading coefficient of the polynomial $c$ is uniquely determined by $\ell$ and $\deg(c)=a-\deg(\ell)$. Hence we have $$\sum_{c \in \FF_q[x] \atop \ell c \in W_q(a)}1=q^{a-\deg(\ell)}$$ and therefore we obtain
$$\sum_{k \in W_q(a)\atop \gcd(k,p)=1} 1 = q^a \sum_{\ell|p} \frac{\mu_q(\ell)}{q^{\deg(\ell)}}.$$ Using the factorization of $p$ we now obtain
\begin{eqnarray*}
\sum_{k \in W_q(a)\atop \gcd(k,p)=1} 1 & = & q^a \sum_{d_1=0}^{\alpha_1} \ldots \sum_{d_r=0}^{\alpha_r} \frac{\mu_q(p_1^{d_1}\cdots p_r^{d_r})}{q^{\deg(p_1^{d_1}\cdots p_r^{d_r})}}\\
& = & q^a \prod_{j=1}^r \sum_{d=0}^{\alpha_j} \frac{\mu_q(p_j^{d_j})}{q^{\deg(p_j^{d_j})}}\\
& = & q^a \prod_{j=1}^r \left(1-\frac{1}{q^{\deg(p_j)}}\right) \\
& \ge & q^a \frac{1}{2^r}.
\end{eqnarray*}
Inserting this result into \eqref{sumT} yields $$T \ge \frac{1}{2^r} \sum_{a=1}^\infty \frac{1}{a \log a} = \infty$$ as claimed.

Now assume that $s \ge 2$. As above we begin by studying the inner sum in \eqref{sumT}. We have
\begin{eqnarray*}
\underbrace{\sum_{k_1 \in W_q(a_1)\atop \gcd(k_1,p)=1} \ldots  \sum_{k_s \in W_q(a_s)\atop \gcd(k_s,p)=1}}_{\gcd(k_1,\ldots,k_s)=1} 1 & = &  \sum_{k_1 \in W_q(a_1)\atop \gcd(k_1,p)=1} \ldots  \sum_{k_s \in W_q(a_s)\atop \gcd(k_s,p)=1} \sum_{\ell | \gcd(k_1,\ldots,k_s)} \mu_q(\ell)\\
& = & \sum_{\ell \in \FF_q[x]\atop \deg(\ell) \le \min(a_1,\ldots,a_s)} \mu_q(\ell) \prod_{i=1}^s \left(\sum_{k_i \in W_q(a_i)\atop {\gcd(k_i,p)=1 \atop \ell|k_i}}1\right).
\end{eqnarray*}
For any factor of the above product we have (we omit the index ``$i$'' for the sake of simplicity)
\begin{eqnarray*}
\sum_{k \in W_q(a)\atop {\gcd(k,p)=1 \atop \ell|k}}1 & = & \sum_{c \in \FF_q[x] \atop { \ell c \in W_q(a) \atop \gcd(\ell c,p)=1}}1 = \left\{
\begin{array}{ll}
0 & \mbox{ if } \gcd(\ell,p)>1,\\
\sum\limits_{c \in \FF_q[x] \atop { \ell c \in W_q(a) \atop \gcd(c,p)=1}}1 & \mbox{ if } \gcd(\ell,p)=1. 
\end{array}\right.
\end{eqnarray*}
Using the same arguments as above it can be shown that $$\sum_{c \in \FF_q[x] \atop { \ell c \in W_q(a) \atop \gcd(c,p)=1}}1= q^{a-\deg(\ell)} A(p),$$ where $A(p):=\prod_{j=1}^r \left(1-q^{-\deg(p_j)}\right) \ge 2^{-r}.$
Hence we obtain
\begin{eqnarray*}
\underbrace{\sum_{k_1 \in W_q(a_1)\atop \gcd(k_1,p)=1} \ldots  \sum_{k_s \in W_q(a_s)\atop \gcd(k_s,p)=1}}_{\gcd(k_1,\ldots,k_s)=1} 1 & = &  \sum_{\deg(\ell) \le \min(a_1,\ldots,a_s) \atop \gcd(\ell,p)=1} \frac{\mu_q(\ell)}{q^{s \deg(\ell)}} q^{a_1+\cdots +a_s} A(p)^s\\
& \ge & q^{a_1+\cdots +a_s} A(p)^s \inf_{x \in \NN} \sum_{\deg(\ell) \le x \atop \gcd(\ell,p)=1} \frac{\mu_q(\ell)}{q^{s \deg(\ell)}}. 
\end{eqnarray*}
We show that 
\begin{equation}\label{lbdB}
B:= \inf_{x \in \NN} \sum_{\deg(\ell) \le x \atop \gcd(\ell,p)=1} \frac{\mu_q(\ell)}{q^{s \deg(\ell)}} \ge \frac{q-1}{4}.
\end{equation}
For any $x \in \NN$ we have 
\begin{eqnarray*}
\sum_{\deg(\ell) \le x \atop \gcd(\ell,p)=1} \frac{\mu_q(\ell)}{q^{s \deg(\ell)}} & = &  \sum_{\deg(\ell) =0 \atop \gcd(\ell,p)=1} 1 + \sum_{1 \le \deg(\ell) \le x \atop \gcd(\ell,p)=1} \frac{\mu_q(\ell)}{q^{s \deg(\ell)}}\\
& = & q-1 +  \sum_{1 \le \deg(\ell) \le x \atop \gcd(\ell,p)=1} \frac{\mu_q(\ell)}{q^{s \deg(\ell)}}.
\end{eqnarray*}
for the last sum we have
\begin{eqnarray*}
\left| \sum_{1 \le \deg(\ell) \le x \atop \gcd(\ell,p)=1} \frac{\mu_q(\ell)}{q^{s \deg(\ell)}} \right| & \le & \frac{1}{q^s} \sum_{\deg(\ell)=1 \atop \gcd(\ell,p)=1} 1 + \sum_{d=2}^\infty \frac{1}{q^{s d}} \sum_{\deg(\ell)=d}1\\
&  \le & \frac{(q-1)^2}{q^s}+ \frac{q-1}{q^{s-1}(q^{s-1}-1)},
\end{eqnarray*}
where we used $\sum\limits_{\deg(\ell)=1 \atop \gcd(\ell,p)=1} 1 \le (q-1)^s$ since $x|p$. Hence it follows that $$\sum_{\deg(\ell) \le x \atop \gcd(\ell,p)=1} \frac{\mu_q(\ell)}{q^{s \deg(\ell)}} \ge (q-1) \left(1-\frac{q-1}{q^s}-\frac{1}{q^{s-1}(q^{s-1}-1)}\right) \ge \frac{q-1}{4}$$ and hence \eqref{lbdB} is shown. 

Hence $$\underbrace{\sum_{k_1 \in W_q(a_1)\atop \gcd(k_1,p)=1} \ldots  \sum_{k_s \in W_q(a_s)\atop \gcd(k_s,p)=1}}_{\gcd(k_1,\ldots,k_s)=1} 1 \ge q^{a_1+\cdots +a_s} A(p)^s \frac{q-1}{4}.$$ Inserting this in \eqref{sumT} we obtain
\begin{eqnarray*}
T & \ge &  A(p)^s \frac{q-1}{4} \sum_{a_1,\ldots ,a_s \atop a_1+\cdots +a_s\not =0}\frac{1}{q^{F(a_1,\ldots,a_s)}} q^{a_1+\cdots +a_s}\\
& \ge & \frac{1}{2^{r s}} \frac{q-1}{4} \sum_{d=1}^\infty \frac{1}{d^s \log d} \sum_{a_1,\ldots,a_s=0 \atop a_1+\cdots+a_s=d}^\infty 1\\
& = & \frac{1}{2^{r s}} \frac{q-1}{4} \sum_{d=1}^\infty \frac{1}{d^s \log d} {s+d-1 \choose d}\\
& \ge & \frac{1}{2^{r s}} \frac{q-1}{4} \frac{1}{(s-1)!} \sum_{d=1}^\infty \frac{1}{d \log d}\\
& = & \infty,
\end{eqnarray*}
where we used that ${s+d-1 \choose d} \ge \frac{d^{s-1}}{(s-1)!}$.
\end{proof}

Now we use Lemma~\ref{le4} and find that the set $\overline{M}$ for our choice of $F$ as in \eqref{defF} and $P$ as in \eqref{defP} has measure $\mu_s(\overline{M})=1$.

Next we consider the finite collection of $s$-tuples $$(\beta_1(k_1,u_1),\ldots,\beta_s(k_s,u_s))$$ for $u_1,\ldots,u_s=0,1,\ldots,J$ but not all equal to 0. Note that each of these $\beta_i(k_i,u_i)$ considered as element of $\ZZ_q[x]$ is 0 or relatively prime to $k_i$ for each $(k_1,\ldots,k_s)$ which is an element from $P$ defined above.

Now we use Lemma~\ref{le3} where we choose $\overline{P}$ as $\overline{P}=P$ from \eqref{defP} and for any choice of $u_1,\ldots,u_s=0,1,\ldots,J$ but not all equal to 0, we choose the $\beta_i(k_i)$ from Lemma~\ref{le3} as $\beta_i(k_i)=\beta_i(k_i,u_i)$. Then for the corresponding set $\widetilde{M}:=\widetilde{M}(u_1,\ldots,u_s)$ of Lemma~\ref{le3} we have $\mu_s(\widetilde{M})=0$.

We set $$M:= \overline{M} \setminus \bigcup_{u_1,\ldots,u_s=0 \atop (u_1,\ldots,u_s)\not=(0,\ldots,0)}^J \widetilde{M}(u_1,\ldots,u_s)$$ and find that $\mu_s(M)=1$. 

Now we make a suitable choice for $\bsf=(f_1,\ldots,f_s)$ and for $\bsk^\ast$. Let $(f_1,\ldots,f_s) \in M$ and let $(\widetilde{k}_1,\ldots,\widetilde{k}_s) \in P$ be such that
$$\nu(\{\widetilde{k}_1 f_1+\cdots +\widetilde{k}_sf_s\}) \le -F(r_1,\ldots,r_s) \le -(r_1+\cdots+r_s)$$ and $$\nu(\{(\widetilde{k}_1+\beta_1(\widetilde{k}_1,u_1))f_1+\cdots + (\widetilde{k}_s+\beta_s(\widetilde{k}_s,u_s))f_s\}) \ge -\frac{r_1+\cdots+r_s}{2},$$ where $r_i=\deg(\widetilde{k}_i)$ and $\widetilde{k}_i=q^{\widetilde{a_i}-1}+\widetilde{\ell}_i$. By the definition of $M$ there are infinitely many such $s$-tuples $(\widetilde{k}_1,\ldots,\widetilde{k}_s)$.

Let $m:=\lfloor F(r_1,\ldots,r_s)\rfloor$ and $N=q^{m-1}$. We analyze the first summand in \eqref{defLambda}: We have $$-w(\widetilde{k}_1,\ldots,\widetilde{k}_s)=\nu(\{\widetilde{k}_1f_1+\cdots+\widetilde{k}_s f_s\}) \le -F(r_1,\ldots,r_s) \le -m$$ and hence $w \ge m$. This means that $G(N,w(\widetilde{k}_1,\ldots,\widetilde{k}_s))=N$ and hence we obtain
$$\left|\left[\prod_{j=1}^s \theta(\widetilde{k}_j)\right] G(N,w(\widetilde{k}_1,\ldots,\widetilde{k}_s))\right| \ge c_6(q,s) \frac{N}{q^{\widetilde{a}_1+\cdots +\widetilde{a}_s}}.$$ 

Now we turn to the second  summand in \eqref{defLambda}. We have 
\begin{eqnarray*}
\lefteqn{-w(\widetilde{k}_1+\beta_1(\widetilde{k}_1,u_1),\ldots,\widetilde{k}_s+\beta_s(\widetilde{k}_s,u_s))}\\
& = &  \nu(\{(\widetilde{k}_1+\beta_1(\widetilde{k}_1,u_1))f_1+\cdots + (\widetilde{k}_s+\beta_s(\widetilde{k}_s,u_s))f_s\})\\
& \ge & -\frac{r_1+\cdots +r_s}{2}
\end{eqnarray*}
and hence $$w(\widetilde{k}_1+\beta_1(\widetilde{k}_1,u_1),\ldots,\widetilde{k}_s+\beta_s(\widetilde{k}_s,u_s)) \le \frac{r_1+\cdots +r_s}{2}.$$ This means that 
\begin{eqnarray*}
|G(N,w(\widetilde{k}_1+\beta_1(\widetilde{k}_1,u_1),\ldots,\widetilde{k}_s+\beta_s(\widetilde{k}_s,u_s))|
& \le & c_7(q) q^w \\
& \le & c_7(q) q^{\frac{r_1+\cdots +r_s}{2}}\\
& = & c_8(q,s) q^{\frac{\widetilde{a}_1+\cdots+\widetilde{a}_s}{2}}
\end{eqnarray*}
and hence we obtain 
\begin{eqnarray*}
\left|\sum_{u_1,\ldots,u_s \ge 0 \atop \exists j:\ u_j >J} \left[\prod_{j=1}^s \theta(\widetilde{k}_j+\beta_j(\widetilde{k}_j,u_j))\right] G(N,w(\widetilde{k}_1+\beta_1(\widetilde{k}_1,u_1),\ldots,\widetilde{k}_s+\beta_s(\widetilde{k}_s,u_s)) \right| \\ \le c_9(q,s) \frac{q^{\frac{\widetilde{a}_1+\cdots+\widetilde{a}_s}{2}}}{q^{\widetilde{a}_1+\cdots+\widetilde{a}_s}}.
\end{eqnarray*}

Altogether we have 
\begin{eqnarray*}
|\Lambda| & \ge &  c_6(q,s) \frac{N}{q^{\widetilde{a}_1+\cdots +\widetilde{a}_s}} - c_9(q,s) \frac{q^{\frac{\widetilde{a}_1+\cdots+\widetilde{a}_s}{2}}}{q^{\widetilde{a}_1+\cdots+\widetilde{a}_s}}-c_5(q,s) \frac{N}{q^{a_1^\ast + \cdots + a_s^\ast}} \frac{1}{q^{J/q}}\\
& \ge & c_{10}(q,s) \frac{N}{q^{\widetilde{a}_1+\cdots +\widetilde{a}_s}}
\end{eqnarray*}
for $J$ large enough and for $\deg(\widetilde{k}_1)+\cdots +\deg(\widetilde{k}_s)$ large enough. Now
\begin{eqnarray*}
\frac{N}{q^{\widetilde{a}_1+\cdots +\widetilde{a}_s}} & \ge & c_{11}(q,s) q^{F(r_1,\ldots,r_s)-r_1-\cdots -r_s} \\
& = & c_{11}(q,s)(r_1+\cdots+r_s) \log(r_1+\cdots+r_s) \\
& \ge & c_{12}(q,s) (\log N)^s \log \log N.
\end{eqnarray*}

From the definition of $\Lambda$ it follows that there exists an $\bsx \in \QQ^s(q^m) \subseteq[0,1)^s$ such that $$|D(\bsx,N)| \ge  c_{13}(q,s) (\log N)^s \log \log N$$ and the proof of Theorem~\ref{th1} is finished. \hfill $\qed$

\begin{small}
\noindent\textbf{Authors' address:}\\
\noindent Gerhard Larcher, Friedrich Pillichshammer\\
Institut f\"{u}r Finanzmathematik, Universit\"{a}t Linz, Altenbergerstr.~69, 4040 Linz, Austria\\
E-mail: \\ \texttt{gerhard.larcher@jku.at},\\ 
\texttt{friedrich.pillichshammer@jku.at}
\end{small}


\begin{thebibliography}{10}
\bibitem{beck} J. Beck: Probabilistic diophantine approximation, I. Kronecker-sequences. Ann. Math. 140: 451--502, 1994.
\bibitem{car} L. Carlitz: The arithmetic of polynomials in a Galois field.  Amer. J. Math. 54: 39--50, 1932.
\bibitem{DP10} J. Dick and F. Pillichshammer: {\it Digital Nets and Sequences. Discrepancy Theory and Quasi-Monte Carlo Integration}. Cambridge University Press, Cambridge, 2010. 
\bibitem{KP12} P. Kritzer and F. Pillichshammer: Low discrepancy polynomial lattice point sets. J. Number Theory 132:  2510--2534, 2012.
\bibitem{kuinie} L. Kuipers and H. Niederreiter: {\it Uniform Distribution of
Sequences}. John Wiley, New York, 1974; reprint, Dover Publications, Mineola, NY, 2006.
\bibitem{L93} G. Larcher: Nets obtained from rational functions over finite fields. Acta Arith. 63: 1--13, 1993.
\bibitem{L95} G. Larcher: On the distribution of an analog to classical Kronecker-sequences. J. Number Theory 52: 198--215, 2995.
\bibitem{LN93} G. Larcher and H. Niederreiter: Kronecker-type sequences and nonarchimedean diophantine approximation. Acta Arith. 63: 380--396, 1993.
\bibitem{LP13} G. Larcher and F. Pillichshammer: A metrical best possible lower bound on the star discrepancy of digital sequences. submitted, 2013.
\bibitem{nie87} H. Niederreiter: Point sets and sequences with small discrepancy. Monatsh. Math.~104: 273--337, 1987.
\bibitem{niesiam} H. Niederreiter: {\it Random Number Generation
and Quasi-Monte Carlo Methods.} SIAM, Philadelphia, 1992.
\bibitem{roth} K.F. Roth: On irregularities of distribution. Mathematica 1: 73--79, 1954.
\bibitem{Schm72distrib} W.M. Schmidt: Irregularities of distribution VII. Acta Arith. 21:  45--50, 1972.
\bibitem{Sp} V. G. Sprind\v{z}uk: {\it Metric Theory of Diophantine Approximations}. Scripta Series in Mathematics. V. H. Winston \& Sons, Washington, D.C.; A Halsted Press Book, John Wiley \& Sons, New York-Toronto, Ont.-London, 1979. 
\bibitem{weyl} H. Weyl: \"Uber die Gleichverteilung von Zahlen modulo Eins. Math. Ann. 77: 313--352, 1916.
\end{thebibliography}
\end{document}